\documentclass[12pt]{amsart}
\usepackage{times}
\usepackage{amsmath}
\usepackage{amsthm}
\usepackage{tikz}
\usetikzlibrary{matrix,arrows}
\newcommand{\amalmon}[1]{\star_{#1}}
\newcommand{\amal}[1]{\ast_{#1}}
\newcommand{\isom}{\cong}
\newcommand{\bC}{\mathbb{C}}
\newcommand{\bN}{\mathbb{N}}
\newcommand{\bZ}{\mathbb{Z}}
\newcommand{\sA}{\mathcal{A}}

\newcommand{\sD}{\mathcal{D}}
\newcommand{\sE}{\mathcal{E}}
\newcommand{\sH}{\mathcal{H}}
\newcommand{\sJ}{\mathcal{J}}
\newcommand{\sP}{\mathcal{P}}
\newcommand{\sL}{\mathcal{L}}

\newcommand{\sR}{\mathcal{R}}
\newcommand{\sS}{\mathcal{S}}
\newcommand{\sT}{\mathcal{T}}

\newcommand{\PI}{\mathrm{PI}}

\newcommand{\cstar}{C^*}

\theoremstyle{plain}
\newtheorem{thm}{Theorem}
\newtheorem{prop}[thm]{Proposition}
\newtheorem{lemma}[thm]{Lemma}

\theoremstyle{remark}

\newtheorem{example}[thm]{Example}

\begin{document}
\title{Amalgams of Inverse Semigroups and C$^*$-algebras}
\author{Allan P. Donsig}
\address{Mathematics Department, University of Nebraska-Lincoln, Lincoln, NE, 68588-0130}
\email{adonsig1@unl.edu}
\author{Steven P. Haataja}
\author{John C. Meakin}
\address{Mathematics Department, University of Nebraska-Lincoln, Lincoln, NE, 68588-0130}
\email{jmeakin@math.unl.edu}
\thanks{This paper is based, in part, on the doctoral dissertation of the second author,
who was advised by the other two authors.  The second author is deceased.}

\begin{abstract}
An amalgam of inverse semigroups $[S,T,U]$ is full if
$U$ contains all of the idempotents of $S$ and $T$.
We show that for a full amalgam $[S,T,U]$,
$\cstar(S \amal{U} T) \cong \cstar(S) \amal{\cstar(U)} \cstar (T)$.
Using this result, we describe certain
amalgamated free products of C$^*$-algebras, including
finite-dimensional $\cstar$-algebras, the Toeplitz algebra, and
the Toeplitz $\cstar$-algebras of graphs.
\end{abstract}
\maketitle

\section{Introduction}
Inverse semigroups are playing an increasingly prominent role in
the theory of C$^*$-algebras. This paper connects certain amalgams
of inverse semigroups and of C$^*$-algebras. Using this connection,
we describe amalgams of various $\cstar$-algebras.

The first work on amalgamated free products of $\cstar$-algebras
that we know of is due to Blackadar ~\cite{Blackadar}.  Shortly
thereafter, Larry Brown noted in~\cite{Brown} that for countable
discrete groups $G$ and $H$ with a common subgroup $K$, $\cstar(G
\amal{K} H) \isom \cstar(G) \amal{\cstar(K)} \cstar(H)$. The
obvious generalization for inverse semigroups is not true, even
for finite inverse semigroups, without some restriction; see, for
instance, Example~\ref{eg:notfull} below. In Section~\ref{S:amal},
we prove an analogous result for full amalgams of discrete inverse
semigroups, namely
\[ \cstar( S \amal{U} T) \isom \cstar(S) \amal{\cstar(U)} \cstar(T). \]

We apply this result to describe the structure of certain amalgams
of $\cstar$-algebras. First, we describe amalgams of
finite-dimensional $\cstar$-algebras over the natural diagonal
matrices in Section~\ref{S:matalg}. These amalgams turn out to be
direct sums of matrix algebras over the $\cstar$-algebras of free
groups.  The ranks of the free groups and the sizes of the matrix
algebras are easily computed using graphs arising from Bass-Serre
theory~\cite{HMM}. These methods extend to direct sums of matrix
algebras over group $\cstar$-algebras.

Section~\ref{S:meg} gives some structural results for amalgams of a
strongly $E^*$-unitary inverse semigroup with itself. These
results allow us to apply work of Khoshkam and Skandalis~\cite{KS}
and of Milan~\cite{Milan} to decompose certain amalgams of
$\cstar$-algebras as either  crossed products or partial crossed
products of abelian $\cstar$-algebras and groups. Specifically,
Section~\ref{S:calgex} shows that a full amalgam of the Toeplitz
algebra with itself is strongly Morita equivalent to a crossed
product of an abelian $\cstar$-algebra and a group, while the
amalgam of a Toeplitz graph $\cstar$-algebra with itself over the
natural diagonal is isomorphic to a partial crossed product of an
abelian $\cstar$-algebra and a group.

We remark that the structure of amalgamated free products of
semigroups or of inverse semigroups is far from understood in
general. For example, it is known that the word problem for an
amalgamated free product $S_1 \amal{U} S_2$ of semigroups (in the
category of semigroups) may be undecidable even if  $S_1, S_2$ and
$U$ are finite semigroups~\cite{Sapir}. On the other hand, the
word problem for an amalgamated free product $S_1 \amal{U} S_2$ of
finite inverse semigroups in the category of inverse semigroups is
decidable \cite{CMP}. It follows from results of Bennett
\cite{Bennett} that the word problem for $S_1 \amal{U} S_2$ is
decidable if $U$ is a full inverse subsemigroup of the inverse
semigroups $S_1$ and $S_2$.

The structure of amalgamated free products of $\cstar$-algebras
has been studied extensively by Pedersen in~\cite{Pedersen},
which also includes an excellent introduction and bibliography.

Next, we review the background we need.
For more information, see~\cite{Howie2}, \cite{Lawson}, or~\cite{Petrich}
for introductions to inverse semigroups;
see~\cite{Davidson} or~\cite{Fillmore}, for example, for more
on $\cstar$-algebras.

Amalgamated free products may be defined in any category by the
standard universal property. Given objects $U$, $S_1$, and $S_2$,
with monomorphisms $i_j : U \to S_j$, $j=1,2$ in some category,
the \textit{free product of $S_1$ and $S_2$, amalgamated over $U$}
is an object $T$ and morphisms $\psi_i : S_i \to T$
with $\psi_1 i_1 = \psi_2 i_2$ so that for
any object $R$ and morphisms $\phi_j : S_j \to R$ with $\phi_1 i_1
= \phi_2 i_2$, there is a unique morphism $\lambda : T \to R$ so
that the following diagram commutes:

\begin{equation} \label{eq:amaldia}
\begin{tikzpicture}[baseline=-.2em,description/.style={fill=white,inner sep=2pt}]
\matrix (m) [matrix of math nodes, row sep=2.5em, column
sep=2.5em,
    text height=1.5ex, text depth=0.25ex]
{ & S_1 \\
 U & & T & & R \\
& S_2 \\ }; \path[->,font=\scriptsize] (m-2-1) edge node[auto] { $
i_1 $ } (m-1-2)
    edge node[auto,swap] { $ i_2 $ } (m-3-2)
(m-1-2) edge node[right=5pt,above=1pt] { $ \psi_1 $ } (m-2-3)
    edge [bend left=15] node[auto] { $ \phi_1 $ } (m-2-5)
(m-3-2) edge node[right=5pt,below=1pt] { $ \psi_2 $ } (m-2-3)
    edge [bend right=15] node[auto,swap] { $ \phi_2 $ } (m-2-5);
\path[->,font=\scriptsize,dotted] (m-2-3) edge
node[above=6pt,left=3pt] {$\lambda$} (m-2-5);
\end{tikzpicture}
\end{equation}
If it exists, the object $T$ is unique up to isomorphism and is
denoted $S_1 \amal{U} S_2$. The tuple $[S_1,S_2,U,i_1,i_2]$ is
called an {\it amalgam} in the category: in all cases of interest
in this paper, the monomorphisms $i_1,i_2$ will be embeddings.
We will often use $[S_1,S_2,U]$ and think of $U$ as contained in $S_1$
and $S_2$.

\subsection*{Inverse Semigroups}\label{subsect:IS}
An inverse semigroup is a semigroup $S$ such that for each $s \in
S$ there exists a unique element $s^{-1} \in S$ such that
$ss^{-1}s=s$ and $s^{-1}ss^{-1}=s^{-1}$.
Every inverse semigroup
$S$ is evidently (von Neumann) regular, i.e., for each $s \in S$
there exists $t \in S$ such that $s = sts$. Inverse semigroups can
be characterized as those regular semigroups whose idempotents
commute~\cite[Theorem~1.1.3]{Lawson}. Inverse semigroups may also
be viewed as an equationally defined class of semigroups with an
involution $s \mapsto s^{-1}$ so that $ss^{-1}s=s$ and
$ss^{-1}tt^{-1}=tt^{-1}ss^{-1}$ for all $s$ and $t$
\cite[Theorem~VIII.1.1]{Petrich}.

We denote the set of
idempotents of an inverse semigroup $S$ by $E(S)$, or $E$, if $S$
is clear: $E(S)$ is a commutative idempotent semigroup, (i.e., a
semilattice) relative to the product in $S$. There is a
natural partial order on an inverse semigroup $S$ defined by $a
\leq b$ (for $a,b \in S$) iff there exists $e \in E(S)$ such that
$a = eb$. The smallest congruence $\sigma$ on $S$ for which
$S/\sigma$ is a group is generated by collapsing this partial
order.
Note that if $S$ has a zero, then $S/\sigma \cong \{0\}$.

An inverse subsemigroup $T$ of an inverse semigroup $S$ is
called a {\it full} subsemigroup of $S$ if it contains all of the
idempotents of $S$, i.e., $E(T) = E(S)$. An amalgam
$[S_1,S_2,U]$ of inverse semigroups is called a {\it full}
amalgam if $U$ is a full inverse subsemigroup of $S_1$ and $S_2$.

It is a non-trivial fact that the category of inverse semigroups
 has the {\it strong amalgamation property}, that is, if
$[S_1,S_2,U,i_1,i_2]$ is an amalgam of inverse semigroups, then in
the notation of the definition above, the morphisms $\psi_i$ are
monomorphisms (embeddings) and $\psi_1(S_1) \cap \psi_2(S_2)$
equals the image of $U$ \cite{Hall}. This property fails in
general in the category of semigroups \cite[p.\ 139]{CP}.

In \cite{HMM}, the authors use Bass-Serre theory to describe the
structure of the maximal subgroups of  $S_1 \amal{U} S_2$ in the
case where $[S_1,S_2,U,i_1,i_2]$ is a full amalgam. We will
use these results in Sections~\ref{S:matalg} and~\ref{S:calgex}.

An inverse semigroup $S$ may or may not have an identity element
$1$ or a zero element $0$. If $S$ has an identity we refer to it
as an inverse monoid.
If $S$ does not have a zero, we may
adjoin one, obtaining the inverse semigroup with zero $S^0 = S
\cup \{0\}$ with the obvious multiplication making $0$ the zero
element.

A representation of an inverse semigroup $S$ is a homomorphism of semigroups
$\rho : S \to B(\sH)$, the bounded operators on a Hilbert space $\sH$, such that $\rho$
sends the inverse operation of the semigroup to the adjoint operation of $B(\sH)$.
Each $T \in \rho(S)$ satisfies $TT^*T=T$ and $T^*TT^*=T^*$ and so
$T$ is a partial isometry in $B(\sH)$.
In fact, every inverse semigroup can be faithfully represented as a semigroup of
partial isometries on some Hilbert space~\cite{DP}.

\subsection*{$\cstar$-Algebras}
One can define $\cstar(S)$ so that it has the universal property that each
representation of $S$ lifts to a unique representation of $\cstar(S)$.
Precisely, there is a monomorphism $i : S \to \cstar(S)$ so that,
for each representation $\rho :S \to B(\sH)$, there is a
unique representation $\tilde{\rho} : \cstar(S) \to B(\sH)$ with
$\tilde{\rho} \circ i = \rho$.
It follows from the uniqueness that if two representations of
$\cstar(S)$ agree on $S$, then they are equal.
For details, see \cite[Section~1]{DP}.
Of course, for a finite inverse semigroup $S$, $\cstar(S)$ is the
complex inverse semigroup algebra, $\bC S$.

For an inverse semigroup $S$ with a zero, $0$, it is natural to restrict
to representations that send $0$ to the zero operator.
If we modify the universal property of $\cstar(S)$ to consider only such
representations, then we obtain the contracted $\cstar$-algebra, $\cstar_0(S)$.
This can be identified with the quotient of $\cstar(S)$ by the ideal
generated by $0$, which is a copy of the complex numbers. 
That is, $\cstar(S) \isom \cstar_0(S)\oplus \bC$. 
We can define $\bC_0 S$ similarly.

Let $\sP$ be a semilattice of projections in a $\cstar$-algebra $\sA$,
that is, $\sP$ is closed under products.
Note that $\sP$ is always commutative, as
two projections in $\sA$ whose product is a projection must commute.
Define $\PI(\sP)$ to be the set of all partial isometries $X$ in $\sA$,
i.e., elements satisfying $X=XX^*X$ and $X^*=X^*XX^*$, such that
(1)~$X^*X, XX^* \in \sP$, and (2)~$X^*\sP X \subseteq \sP$ and  $X\sP X^* \subseteq \sP$.
Observe that if $\sA$ is unital and $1_A \in \sP$, then Condition~(2) gives
$XX^*=X 1_A X^* \in \sP$ and $X^*X \in \sP$ similarly, so we can omit 
Condition~(1) from the definition in this case.

\begin{prop}\label{prop:PIprop}
If $\sP$ is a semilattice of projections in a $\cstar$-algebra $\sA$,
Then $\PI(\sP)$ is an inverse semigroup with idempotents $\sP$.
Also, if $\sS \subset \sA$ is an inverse semigroup with $E(\sS)\subseteq \sP$,
then $\sS \subseteq \PI(\sP)$.
\end{prop}

\begin{proof}
If $X,Y \in \PI(\sP)$, then $XY$ is a partial isometry, as $X^*X$ and
$YY^*$ are in $\sP$ and so commute. 
Clearly, $(XY)^* \sP XY = Y^*(X^*\sP X)Y \subseteq \sP$
and $(XY)^*XY=Y^*(X^*X)Y$ is in $\sP$, as $X^*X \in \sP$
and $Y^*\sP Y\subseteq \sP$. 
Verifying that $XY \sP (XY)^* \subseteq \sP$ and $XY(XY)^* \in \sP$
are similar, so $XY \in \PI(\sP)$.
If $X \in \PI(\sP)$, so is $X^*$, and $X^*$ is an inverse for $X$.
Finally, if $X$ is an idempotent in $\PI(\sP)$, then it is easy to
check that $X$ is a projection and hence $X=XX^* \in \sP$. 
As a regular semigroup whose idempotents commute, $\PI(\sP)$ is 
an inverse semigroup.

For $\sS$ as above, each $X \in \sS$ satisfies $XX^*X=X$, $X^*XX^*=X^*$,
conjugates $E(\sS)$ into itself, and has both $X^*X$ and $XX^*$ in $E(\sS)$.
Thus $\sS \subseteq \PI(\sP)$.
\end{proof}

In particular, it follows that if $\Psi : S \to \sA$ is a representation
of an inverse semigroup in a $\cstar$-algebra, then $\Psi(S) \subseteq
\PI(\Psi(E))$.

\section{Amalgams} \label{S:amal}

Before turning to our main theorem, we first point out a related
result.

\begin{thm}\label{thm:Calgamal}
Suppose $S=[S_1,S_2,U]$ is an amalgam of inverse semigroups with $U$ a full inverse
subsemigroup of both $S_1$ and $S_2$.  Then, in the category of complex algebras,
\[ \bC S \cong \bC S_1 \amal{\bC U} \bC S_2. \]
\end{thm}

\begin{example} \label{eg:notfull}
The conclusion of Theorem~\ref{thm:Calgamal} is not true without some condition on the amalgam.
Let $S$ and $T$ be different copies of the two element semilattice, i.e.,
$S=\{e,0\}$ with $e^2=e$ and all other products equal to $0$ and $T=\{f,0\}$ is similar.
Letting $U=\{0\}$, we see that the inverse semigroup amalgam, $S \amal{U} T$,
has four elements $e,f,ef$, and $0$ and $\bC(S \amal{U} T)=\bC^4$.
We also have $\bC S=\bC^2$, $\bC T=\bC^2$ and
$\bC U=\bC$. However, the existence of inverse semigroup
homomorphisms from $\bC S$ and $\bC T$ into a complex algebra does
not force the images of $e$ and $f$ to commute. Thus, in general,
there is no homomorphism from $\bC(S \amal{U} T)\cong\bC^4$ to $\bC S
\amal{\bC U} \bC T$.
\end{example}

Another complicating fact is that the functor $S \mapsto \bC S$ from inverse semigroups to
complex algebras behaves badly with respect to colimits.
The difficulty is that the multiplicative semigroup of $\bC S$ need not be an
inverse semigroup.
The construction of a complex algebra can be performed on an arbitrary semigroup, though,
and we can use this to prove the result for full amalgams.

\begin{proof}[\protect{Proof of Theorem~\ref{thm:Calgamal}}]
Consider the amalgamated free product, in the category of semigroups, of inverse
semigroups $S_1$ and $S_2$ over the inverse semigroup $U$, which we denote by
$S_1 \amalmon{U} S_2$.
Morphisms in the category of inverse semigroups are just semigroup 
morphisms~\cite[p.~30]{Lawson}.
The functor that sends a semigroup $M$ to $\bC M$ has a right adjoint given by the
forgetful functor (forget everything in $\bC M$ except the multiplication).
It follows that this functor preserves
colimits~\cite[Dual of Theorem V.5.1]{MacLane}.
Thus the Diagram~\eqref{eq:amaldia} does lift to the category of complex algebras, but from
the category of semigroups, not that of inverse semigroups.
That is,
\[\bC(S_1 \amalmon{U} S_2) = \bC S_1 \amal{\bC U} \bC S_2. \]
Finally, \cite[Theorem~2]{Howie} asserts that for $U$ a full inverse subsemigroup of 
both $S_1$ and $S_2$, then $S_1 \amalmon{U} S_2$ is an inverse semigroup.
Thus, $S_1 \amalmon{U} S_2 = S_1 \amal{U} S_2$ and so
\[ \bC S_1 \amal{\bC U} \bC S_2 = \bC(S_1 \amal{U} S_2), \]
as required.
\end{proof}

Given a full amalgam of inverse semigroups $[S_1,S_2,U,i_1,i_2]$, the inclusions $i_j : U \to S_j$
induce inclusions $I_j: \cstar(U) \to \cstar(S_j)$.
Thus, we have an associated amalgam of $\cstar$-algebras $[\cstar(S_1), \cstar(S_2), \cstar(U)]$.
We will always assume that the inclusions of this amalgam are induced by the inverse
semigroup inclusions.

\begin{thm}\label{thm:Cstaramal}
Suppose that $[S_1,S_2,U]$ is a full amalgam of inverse semigroups.
Then
\[ \cstar(S_1 \amal{U} S_2) \isom \cstar(S_1) \amal{\cstar(U)} \cstar(S_2) \]
and, if $U$ has a zero, then
\[ \cstar_0(S_1 \amal{U} S_2) \isom \cstar_0(S_1) \amal{\cstar_0(U)} \cstar_0(S_2). \]
\end{thm}

\begin{proof}
We show $\cstar(S_1 \amal{U} S_2)$ has the universal property of the $\cstar$-algebraic
amalgam $\cstar(S_1) \amal{\cstar(U)} \cstar(S_2)$ and so is isomorphic to it.
Precisely, if $i_j : U \to S_j$, $\psi_j : S_j \to S_1 \amal{U} S_2$ are the canonical injections, then
we use the lifts $I_j : \cstar(U) \to \cstar(S_j)$ and $\Psi_j : \cstar(S_j) \to \cstar(S_1 \amal{U} S_2)$.

Let $\sA$ be a $\cstar$-algebra and suppose there are $*$-homomorphisms
$\Phi_j : \cstar(S_j)\to\sA$,
that agree on $\cstar(U)$, that is, $\Phi_1 \circ I_1 = \Phi_2 \circ I_2$.
We will find an inverse semigroup and homomorphisms from each $S_j$ into that
inverse semigroup that induce $\Phi_j$.

Let $\sP$ be the image of $E(U)$ under $\Phi_j \circ I_j$.
As $E(S_j)=I_j(E(U))$, it follows from Proposition~\ref{prop:PIprop} that
$\Phi_j(S_j) \subseteq \PI(\sP)$ for $j=1$ and $j=2$.
Let $\phi_j : S_j \to \PI(\sP)$ be the restriction of $\Phi_j$ to $S_j$.
Thus we have the Diagram~\eqref{eq:amaldia} in the category of inverse
semigroups, with $R=\PI(\sP)$ and $T=S_1 \amal{U} S_2$.

By the universal property, we have a unique map $\lambda : S_1 \amal{U} S_2 \to \PI(\sP)$
that makes the diagram commute.
Lifting $\lambda$ gives a unique map $\eta$ from $\cstar(S_1\amal{U} S_2)$
into $\cstar(\PI(\sP))$.
The inclusion map from $i : \PI(\sP) \to \sA$ is a representation and
so lifts to a unique map $\zeta: \cstar(\PI(\sP)) \to \sA$.
Letting $\Lambda = \zeta \circ \eta$, we have a map from $\cstar(S_1\amal{U} S_2)$
into $\sA$.

For $j=1,2$, we have
$\Lambda \circ \Psi_j |_{S_j} = \zeta |_{S_j} \circ (\lambda \circ
\psi_j) = i \circ \phi_j = \Phi_j |_{S_j}$.
Since a representation of $\cstar(S_j)$ is determined by its action on
$S_j$, $\Lambda \circ \Psi_j = \Phi_j$.
That is, the following diagram commutes:

\begin{equation}\label{eq:cstardia}
\begin{tikzpicture}[baseline=-.2em,description/.style={fill=white,inner sep=2pt}]
\matrix (m) [matrix of math nodes, row sep=2.5em, column sep=2.5em,
    text height=1.5ex, text depth=0.25ex]
{ & \cstar(S_1) \\
 \cstar(U) & & \cstar(S_1 \amal{U} S_2) & & \sA \\
& \cstar(S_2) \\ };
\path[->,font=\scriptsize]
(m-2-1) edge node[auto] { $ I_1 $ } (m-1-2)
    edge node[auto,swap] { $ I_2 $ } (m-3-2)
(m-1-2) edge node[right=5pt,above=1pt] { $ \Psi_1 $ } (m-2-3)
    edge [bend left=15] node[auto] { $ \Phi_1 $ } (m-2-5)
(m-3-2) edge node[right=5pt,below=1pt] { $ \Psi_2 $ } (m-2-3)
    edge [bend right=15] node[auto,swap] { $ \Phi_2 $ } (m-2-5);
\path[->,font=\scriptsize,dotted]
(m-2-3) edge node[above=6pt,left=3pt] {$\Lambda$} (m-2-5);
\end{tikzpicture}
\end{equation}

To see that $\Lambda$ is unique, suppose that replacing $\Lambda$ with
$\mu : \cstar(S_1\amal{U} S_2) \to \sA$ in this diagram also makes it commute.
Then $\mu \circ \Psi_j$ and $\Lambda \circ \Psi_j$
agree on $S_j$, for $j=1,2$ and so $\mu$ and $\Lambda$ agree on a generating
set of $S_1 \amal{U} S_2$ and so agree on $S_1 \amal{U} S_2$.
But this implies $\mu=\Lambda$, as required.
Thus, $\cstar(S_1 \amal{U} S_2)$ has the universal property for amalgamated free products
of $\cstar$-algebras and so is isomorphic to $\cstar(S_1) \amal{\cstar(U)} \cstar(S_2)$.

To obtain the result for the contracted algebras, one can either
repeat the above proof restricting to representations that take $0$ to $0$,
or apply the first result and quotient out on both sides by the ideals
associated to the common zero.  We outline the latter approach.
Consider the following commuting square:
\begin{equation}\label{eq:contracalg}
\begin{tikzpicture}[baseline=-.2em,description/.style={fill=white,inner sep=2pt}]
\matrix (m) [matrix of math nodes, row sep=2.5em, column sep=2.5em,
    text height=1.5ex, text depth=0.25ex]
{ & \cstar_0(S_1) \\
 \cstar_0(U) & & \cstar_0(S_1 \amal{U} S_2) \\
& \cstar_0(S_2) \\ };
\path[->,font=\scriptsize]
(m-2-1) edge node[auto] { $ I'_1 $ } (m-1-2)
    edge node[auto,swap] { $ I'_2 $ } (m-3-2)
(m-1-2) edge node[right=5pt,above=1pt] { $ \Psi'_1 $ } (m-2-3)
(m-3-2) edge node[right=5pt,below=1pt] { $ \Psi'_2 $ } (m-2-3);
\end{tikzpicture}
\end{equation}
where the primed maps are the appropriate lifts of the $i_j$ and $\psi_i$, as above.
Adding a copy of $\bC$ to each contracted $\cstar$-algebra and extending the
primed maps by mapping $\bC$ to $\bC$ gives the commuting square in Diagram~\eqref{eq:cstardia}.
Given $\Phi'_j : \cstar_0(S_j) \to \sA$, we can define 
$\Phi_j : \cstar(S_j) \to \sA\oplus\bC$ by mapping the copy of $\bC$ associated
to the zero of $S$ to the copy of $\bC$ in the codomain algebra.
The result above gives a unique map $\Lambda : \cstar(S_1 \amal{U} S_2) \to \sA\oplus\bC$
and, identifying $\cstar(S_1 \amal{U} S_2)$ with $\cstar_0(S_1 \amal{U} S_2) \oplus\bC$,
then one can show that if $\Lambda'=\Lambda |_{\cstar_0(S_1 \amal{U} S_2)}$, then
the range of $\Lambda'$ is contained in $\sA$, and $\Lambda'$ is the unique
map making the appropriate diagram based on~\eqref{eq:contracalg} commute.
\end{proof}

\section{Amalgams of Finite-Dimensional $\cstar$-Algebras} \label{S:matalg}

As an application, we use Theorem~\ref{thm:Cstaramal} to describe amalgams of
finite-dimensional $\cstar$-algebras, i.e., direct sums of matrix algebras over $\bC$,
over the diagonal matrices.
These methods easily extend to amalgams of direct sums of matrix algebras over
(discrete) group $\cstar$-algebras.

Given a group $G$ and a natural number $n$, we define the Brandt inverse semigroup $B_n(G)$
as the set $\{ (i,g,j) : 1 \le i,j \le n, g \in G \}$ together with $0$
where we define the product of $0$ with any element to be $0$ and
the product of $(i,g,j)$ and $(k,h,l)$ to be $(i,gh,l)$ if $j=k$ and $0$ otherwise.
If $G$ is the trivial group, we use $B_n$ for $B_n(G)$; this is called a
combinatorial Brandt inverse semigroup.
Notice that $B_n$ can be identified with the matrix units of $M_n=M_n(\bC)$, together with
the zero matrix.  Further, $\bC B_n=M_n \oplus \bC$, $\bC_0 B_n=M_n$, and
$\cstar(B_n(G))=M_n(\cstar(G))\oplus \bC$.

Given two semigroups $S$ and $T$, each with a zero $0$, the $0$-direct union of $S$ and $T$
is $S \amal{\{0\}} T$.
If $S$ is the $0$-direct union of finitely many combinatorial Brandt inverse semigroups $B_{n_1}, \ldots, B_{n_k}$,
then $\cstar_0(S)=M_{n_1} \oplus \cdots \oplus M_{n_k}$.
Since all finite-dimensional $\cstar$-algebras are finite direct sums of matrix algebras,
we can identify all finite dimensional $\cstar$-algebras as $\cstar$-algebras of inverse semigroups.

Suppose that $P=\oplus_{i=1}^r M_{m_i}$ and $Q=\oplus_{i=1}^s M_{n_i}$, where
$\sum_i m_i = \sum_i n_i$.
Using $N$ for this common sum, we identify $\bC^N$ with a natural abelian subalgebra of $P$ and $Q$,
namely the diagonal matrices.
We can describe $P \amal{\bC^N} Q$ by recognizing $P$ and $Q$ as $\cstar$-algebras of inverse
semigroups as described above.
If $S$ is the $0$-direct union of $B_{m_1},\ldots,B_{m_r}$, $\cstar_0(S)$ is $P$.
Similarly, $\cstar_0(T)$ is $Q$ for $T$ the $0$-direct union of $B_{n_1},\ldots,B_{n_s}$.
Moreover, $\bC^N=\cstar_0(U)$ for $U=E(S)=E(T)$.
Thus, by Theorem~\ref{thm:Cstaramal},
\[ P \amal{\bC^N} Q = \cstar_0(S \amal{U} T). \]
We apply the results of~\cite{HMM} to describe the maximal subgroups of this amalgam
of inverse semigroups. 
We need one of the standard Green's relations for inverse semigroups:
the $\sJ$-relation on a semigroup  is defined by $u \sJ v$ iff $u$ and
$v$ generate the same principal two sided ideal of the semigroup
\cite[Section~3.2]{Lawson}.
The non-zero $\sJ$-classes of the semigroups $S$ and $T$
correspond to the summands of $P$ and $Q$. As $S$ and $T$ have
trivial maximal subgroups, the construction of~\cite[p.~46]{HMM} gives a
graph of groups with trivial vertex and edge groups, that is, a graph.
This graph has $r+s$ vertices, one for each summand of
$P$ and $Q$, and $N$ edges, one for each matrix unit in $\bC^N$.
Moreover, the edge associated to a non-zero idempotent $e \in U$
connects the vertices associated to the summands of $P$ and $Q$
containing $e$.

Let $W_1,\ldots,W_p$ be the components of this graph.
For each $W_i$, let $k_i$ be the number of edges in $W_i$ and $q_i$ be the number
of edges left over after removing a spanning tree from $W_i$.
Each $W_i$ corresponds to a non-zero $\sJ$-class in $S \amal{U} T$ and the maximal subgroup of
that $\sJ$-class is the free group of rank $q_i$, $F_{q_i}$.
Thus $S \amal{U} T$ is the $0$-direct union of $B_{k_1}(F_{q_1}), \ldots, B_{k_p}(F_{q_p})$.
For more details, see Example~3 of~\cite{HMM} and the subsequent
discussion in~\cite{HMM}.
Summarizing, we have the following result.

\begin{thm}\label{thm:FD}
If $P=\oplus_{i=1}^r M_{m_i}$ and $Q=\oplus_{i=1}^s M_{n_i}$, where
$\sum_i m_i = \sum_i n_i=N$, then
\[ P \amal{\bC^N} Q \isom \bigoplus_{i=1}^p M_{k_i}\bigl(\cstar(F_{q_i})\bigr), \]
where $p$, $k_1,\ldots,k_p$, $q_1,\ldots,q_p$ are obtained from the graph above.
\end{thm}

For example, if $P=M_3\oplus M_3 \oplus M_2$ and $Q=M_2\oplus M_1 \oplus M_2 \oplus M_3$,
then the inverse semigroups are $B_3 \amal{\{0\}} B_3 \amal{\{0\}} B_2$
and $B_2 \amal{\{0\}} B_1 \amal{\{0\}} B_2 \amal{\{0\}} B_3$.
The resulting graph is
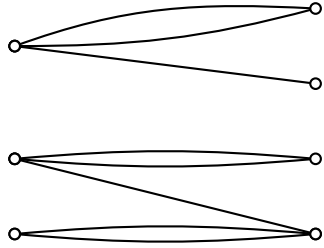
\tikzstyle{every node}=[circle,draw, fill=white,
                        inner sep=0pt, minimum width=4pt]
\[
\begin{tikzpicture}[thick]
   \draw (0,0) node {} to [out=5,in=175] (4,0);
   \draw (0,0) node {} to [out=-5,in=185] (4,0) node {};
   \draw (0,1) node {} --(4,0) node{};
   \draw (0,1) node {} to [out=5,in=175] (4,1);
   \draw (0,1) node {} to [out=-5,in=185] (4,1) node {};
   \draw (0,2.5) node {} to [out=0,in=195] (4,3);
   \draw (0,2.5) node {} to [out=20,in=177] (4,3) node {};
   \draw (0,2.5) node {} --(4,2) node{};

\end{tikzpicture}
\]

\medskip
\noindent
and so we have two components, with $k_1=3$, $q_1=1$, $k_2=5$ and $q_2=2$.
Thus, $ P \amal{\bC^8} Q \isom M_3(\cstar(\bZ))\oplus M_5(\cstar(F_2))$.

Of course, Theorem~\ref{thm:FD} immediately gives the $K$-theory of such
amalgams, first obtained by McClanahan in~\cite{McC}.
That $K_0(M_k(\cstar(F_q)))=\bZ$ and $K_1(M_k(\cstar(F_q)))=\bZ^q$ follow from
the stability of $K$-groups and the short exact sequence on page~83 of~\cite{Cuntz}.
Hence we obtain
\[ K_0(P \amal{\bC^N} Q) = \bZ^p, \qquad K_1(P \amal{\bC^N} Q) = \bZ^q, \]
where $q=q_1+\cdots +q_p$.
Haataja has shown (see~\cite[Section~4.3]{SHDiss}) that this agrees with
McClanahan's procedure for the computation of the
$K$-groups,~\cite[Proposition~7.1]{McC}.

Of course, these methods also apply to amalgams of matrix algebras over
group $\cstar$-algebras, as these are the $\cstar$-algebras of inverse
semigroups of the form $B_n(G)$, for a fixed group $G$.
We leave the details to the interested reader.

\section{Some Special Amalgams}  \label{S:meg}

In this section, we look at the structure of special amalgams and
describe the universal group of an inverse semigroup with zero, which is a
suitable generalization of the maximal group homomorphic image.
This enables us, in the next section, to describe certain special
amalgams of $\cstar$-algebras.

A {\it special amalgam} of inverse monoids is an amalgam $[S,S,U]$
of two copies of $S$ over a common inverse submonoid $U$. More
precisely, it is an amalgam $[S_1,S_2,U,i_1,i_2]$ together with an
isomorphism $\theta : S_1 \to S_2$ such that $i_2 = {\theta} \circ
i_1$.  If $G$ is a group, then the amalgamated free product $G
\amal{U} G$ is referred to as a ``double" of the group $G$.
The terminology ``special amalgam" comes from universal algebra, where
this concept has been well studied. 

We need more of the Green's relations:
$a \sH b$ if and only if $aa^{-1}=bb^{-1}$ (i.e.,
$a \sR b$) and $a^{-1}a=b^{-1}b$ (i.e., $a \sL b$).
For a full treatment of these relations, see, for example, 
\cite[Section~3.2]{Lawson}.

\begin{lemma}\label{lemma:specialH}
Let $U$ be a full inverse submonoid of an inverse monoid $S$ and
consider the special amalgam $[S,S,U,i_1,i_2]$ with associated
isomorphism $\theta : S \to S$. Then $ a \,\sH \, \theta(a)$ and
$ab \, \sH \, \theta(a)\theta(b) \, \sH \, a\theta(b) \, \sH \,
\theta(a)b$ in $S \amal{U} S$ for all $a,b \in S$.
\end{lemma}

\begin{proof}
Since $e$ is identified with $\theta(e)$ in $S \amal{U} S$ for all idempotents $e
\in E(S)$, it follows that $aa^{-1} = \theta(a)\theta(a)^{-1}$ and
$a^{-1}a = \theta(a)^{-1}\theta(a)$ in $S \amal{U} S$ and hence $a
\,  \sH \, \theta(a)$ in $S \amal{U} S$. It follows that $ab \,
\sH \, \theta(a)\theta(b)$ for all $a,b \in S$. Also, $ab \, \sR
\, abb^{-1}a^{-1}$, which is identified with
$a\theta(b)\theta(b^{-1})a^{-1}$ in $S \amal{U} S$, so $a\theta(b)
\, \sR \, ab$. Similarly $a\theta(b) \, \sL \, ab$ in $S \amal{U}
S$, so $ab \, \sH \, a\theta(b)$. The proof that $ab \, \sH \,
\theta(a)b$ is similar.
\end{proof}

A subset $U$ of a semigroup $S$ is called a {\it unitary subset}
of $S$ if, whenever either $us \in U$ or $su \in U$ for some $s
\in S, u \in U$, then $s \in U$. An inverse semigroup $S$ is
called {\it $E$-unitary} if $E(S)$ is a unitary subsemigroup of
$S$: equivalently, if $a \geq e$ for some $a \in S, \, e \in
E(S)$, then $a \in E(S)$.  The inverse semigroup $S$ is said to be
{\it $F$-inverse} if each $\sigma$-class has a maximum element in
the natural partial order: every $F$-inverse semigroup is
$E$-unitary. If $S$ has a zero, these concepts can be modified to
yield the concept of a $0$-$E$-unitary inverse semigroup (also
referred to as an $E^*$-inverse semigroup), namely, $E(S)-\{0\}$ is a
unitary subset and the concept of a $0$-$F$-inverse semigroup (also referred 
to as an $F^*$-inverse semigroup), namely, each non-zero element of $S$ is below
a unique maximal element in the natural partial order.

For an inverse semigroup $S$ with zero, consider pairs $(G,\phi)$
where $G$ is a group and $\phi : S \to G^0$ is a $0$-morphism,
that is, $\phi^{-1}(0)=\{0\}$ and $\phi(ab)=\phi(a)\phi(b)$ whenever
$ab \ne 0$. (We use $G^0$ for $G\cup\{0\}$ with the obvious
multiplication and, for a group morphism $\alpha : G \to H$, we
use $\alpha^0$ for the $0$-morphism from $G^0$ to $H^0$ that sends $0$ to $0$
and agrees with $\alpha$ on $G$.)
There is a largest group, the {\it universal group $G(S)$} of $S$,
with this property; that is, $(G(S),\phi)$ has the property that if
$\tau : S \to H^0$ is a $0$-morphism, then there is a group morphism
$\beta : G(S) \to H$ so that $\beta^0 \circ \phi = \tau$.
If $S^0$ is $S$ with a zero adjoined, then $G(S^0)$ coincides with
$S/\sigma$, the maximal group homomorphic image of $S$.

Let $\phi: S \to G^0$ be a $0$-morphism from $S$ to $G^0$ for some
group $G$, as in the definition above.
Following \cite{mcalister}, consider $\hat{S}$, the inverse semigroup
given by $\{(s,g): g = \phi(s)$ if $s \neq 0\} \cup \{(0,g): g \in G\}$
with the obvious multiplication.
The maximal group image of $\hat{S}$ is $G$, with the map given by
projection onto the second element of each ordered pair.
Moreover, $S$ and $\hat{S}$ have the same semilattice of idempotents,
and $S$ is the Rees quotient $S \cong \hat{S}/I$ where $I$ is
the ideal $I = \{(0,g): g \in G\}$.

\begin{prop}\label{prop:universal}
For $[S,T,U]$ an amalgam of inverse monoids with a common zero in $U$,
$G(S \amal{U} T) = G(S) \amal{G(U)} G(T)$.
\end{prop}

The analogous result for the maximal group images, i.e., that
\begin{equation}\label{eq:amalsigma}
(S \amal{U} T)/\sigma = (S/\sigma) \amal{U/\sigma} (T/\sigma),
\end{equation}
well-known and the proof strategy below is a natural adaptation of the
proof of~\eqref{eq:amalsigma}.
In fact, since $G(S^0)=S/\sigma$, Equation~\eqref{eq:amalsigma} follows from
Proposition~\ref{prop:universal}.

\begin{proof}
Let $A$ be the amalgam of $[S,T,U]$ in the category of inverse monoids
and let $K$ be the amalgam of $[G(S),G(T),G(U)]$ in the category of groups.
We need to construct the $0$-morphism that is the dotted arrow in the following
diagram:

\begin{center}
\tikzstyle{every node}=[fill=white]
\begin{tikzpicture}
  \matrix (m) [matrix of math nodes, row sep=1.5em,
    column sep=1.5em]{
      & G(S)^0 &   & K^0 \\
    S &      & A & \\
      & G(U)^0 &   & G(T)^0 \\
    U &      & T & \\};
  \path[-stealth]
    (m-3-2) edge (m-1-2)
            edge (m-3-4)
    (m-1-2) edge (m-1-4)
    (m-3-4) edge (m-1-4)
    (m-2-3) edge [densely dotted] (m-1-4)
    (m-2-1) edge (m-1-2)
    (m-4-3) edge (m-3-4)
    (m-4-1) edge (m-3-2)
            edge (m-2-1)
        edge (m-4-3)
    (m-2-1) edge [-,line width=6pt,draw=white] (m-2-3)
            edge (m-2-3)
    (m-4-3) edge [-,line width=6pt,draw=white] (m-2-3)
            edge (m-2-3);
\end{tikzpicture}
\end{center}
Observe that the front square involves inverse semigroup morphisms, the back
square has group morphisms (with zeros added) and the diagonal arrows are
$0$-morphisms.

We have two $0$-morphisms: $\alpha : S \to K^0$, the composition of the maps
$S \to G(S)^0$ and $G(S)^0 \to K^0$, and $\beta : T \to K^0$ defined similarly.
We define a map $\gamma : A \to K^0$ by sending $0$ to $0$ and sending a non-zero
word $s_1 t_1 s_2 \cdots s_n t_n$, with $s_i \in S$, $t_i \in T$, to
\[ \alpha(s_1) \beta(t_1) \alpha(s_2) \cdots \alpha(s_n) \beta(t_n). \]
Notice that since the zero is common to both $S$ and $T$, if $s_1t_1\ldots t_n \ne 0$,
then no subword can equal zero, and so $\gamma(A\backslash\{0\})\subseteq K$.

To show that $\gamma$ is well-defined, we show that $\gamma$ respects the equations
that define an inverse semigroup, as given on page~\pageref{subsect:IS}.
It is easy to see that if $v$ and $w$ are non-zero words in the elements of $S$ and $T$
with $vw$ non-zero, then $\gamma(vw)=\gamma(v)\gamma(w)$.
So $\gamma$ respects the relations that impose associativity.
Since $\alpha$ and $\beta$ agree on $U$, the image under $\gamma$ of a word does not depend
on how we regard an element of $U$ as in either $S$ or $T$.

If $w$ is a word in the non-zero elements of $S$ and $T$, let $w^{-1}$
be the word in the inverse elements, written in reverse order.
This is clearly an involution on such words.
It is easy to see that, for $w$ as above, $ww^{-1}= s_1 t_1 \cdots s_n t_n t_n^{-1} s_n^{-1} \cdots s_1^{-1}$.
Using $\beta(t_n t_n^{-1})=1_K$, the identity of $K$, and so on, we obtain
$\gamma( ww^{-1})=1_K$.
Thus $\gamma$ respects the equations that define the inverse semigroup $S\amal{U} T$
and so is a well-defined map.

We have already observed that $\gamma(vw)=\gamma(v)\gamma(w)$ for words $v$
and $w$ with $vw\ne 0$, so $\gamma$ is a $0$-morphism.
By the construction of $\gamma$, the two squares, one involving $S$, $A$, $G(S)^0$,
and $K^0$, and the other involving $T$, $A$, $G(T)^0$, and $K^0$, each commute.

We will show that $(K,\gamma)$ is the universal group for $A$.
Suppose that $\psi : A \to H^0$ is a $0$-morphism, where $H$ is some group.
We have $0$-morphisms from $S$ to $H^0$, and from $T$ to $H^0$, given by
composition of $\psi$ with the maps in the pushout diagram.
and these maps agree on $U$.
By the universal properties of $G(S)$ and $G(T)$, we have group morphisms $G(S) \to H$
and $G(T) \to H$ that agree on $G(U)$.
By the universal property of $K$, these maps give a map $\tau: K \to H$.
Using the commuting triangles and squares, $\tau^0 \circ \gamma$ agrees with $\psi$
when restricted to either $S$ or $T$.
Since $A$ is the amalgam of $S$ and $T$, it follows that $\psi=\tau^0 \circ \gamma$.

Suppose that $\sigma : K \to H$ is another $0$-morphism satisfying
$\sigma^0\circ\gamma = \psi = \tau^0 \circ \gamma$.
Since $\sigma$ and $\tau$ agree on $\gamma(A)$ in $K^0$, they agree on the images
of $S$ and $T$ under $\gamma$ composed with the natural inclusions.
By the universal properties of $G(S)$ and $G(T)$, $\sigma$ and $\tau$ agree
on the images of $G(S)$ and $G(T)$ in $K$.
But these images determine maps on $K$, and so $\sigma=\tau$.
\end{proof}

The following fact was proved by Bennett.

\begin{prop}[\protect{\cite[Corollary~9]{Bennett2}}]\label{prop:E-unitary}
Let $U$ be a full unitary inverse submonoid of the inverse monoid
$S$. Then $S \amal{U} S$ is $E$-unitary iff $S$ is $E$-unitary.
\end{prop}

If $\phi : S \to G(S)^0$ above also satisfies
$\phi^{-1}(1_G) = E(S)-\{0_S\}$, then we say that $S$ is {\it
strongly $E^*$-unitary}. Strongly $E^*$-unitary inverse semigroups
are precisely Rees quotients of $E$-unitary inverse semigroups;
see Section 3 of \cite{mcalister}.
Such semigroups are $E^*$-unitary, but there are $E^*$-unitary inverse
semigroups that are not strongly $E^*$-unitary ~\cite[p.~22]{BFG}.
We refer the reader to Lawson's book~\cite{Lawson} and
his paper~\cite{Lawson01} for more information about these concepts
and the important role that they play in the theory of inverse semigroups.

We use Bennett's result to establish the following fact about special
amalgams of strongly $E^*$-unitary inverse semigroups.

\begin{lemma}\label{lemma:strongamal}
Let $S$ be a strongly $E^*$-unitary inverse semigroup with semilattice
$E = E(S)$. Then $S \amal{E} S$ is strongly $E^*$-unitary.
\end{lemma}

\begin{proof}
With $\hat{S}$ as above, it follows that $\hat{S}$ is an $E$-unitary cover
of $S$ (i.e., it is $E$-unitary and the natural map that sends $(s,g)$ to $s$
if $s\neq 0$ and $(0,g)$ to $0$ is an idempotent-separating map
from $\hat{S}$ onto $S$).
{}From Proposition \ref{prop:E-unitary}, it follows that
$\hat{S} \amal{E} \hat{S}$ is $E$-unitary.

Let $\theta: S \to S$ be the isomorphism in the construction of
the special amalgam $S \amal{E} S$.
Every non-zero element of $S \amal{E} S$ may be expressed (not
uniquely) in the form $s_1\theta(t_1)s_2\theta(t_2) \ldots
s_n\theta(t_n)$ for some non-zero elements $s_i,t_i \in S$. From
Lemma \ref{lemma:specialH} it follows that
$s_1\theta(t_1)s_2\theta(t_2) \ldots s_n\theta(t_n) \neq 0$ in $S
\amal{E} S$ iff $s_1t_1s_2t_2 \ldots s_nt_n \neq 0$ in $S$. Also,
any sequence of elementary transitions that transforms a non-zero
element $s_1t_1 \ldots s_nt_n$  to an equivalent element $s'_1t'_1
\ldots s'_mt'_m$ in $S \amal{E} S$ may be replaced by an obvious
sequence that transforms the corresponding elements in $\hat{S}
\amal{E} \hat{S}$.
We use $J$ for the set of elements in $\hat{S} \amal{E} \hat{S}$
of the form
\begin{equation}\label{eq:1}
 (s_1,\phi(s_1))(\theta(t_1), \phi(\theta(t_1))) \ldots
    (s_n,\phi(s_n))(\theta(t_n), \phi(\theta(t_n)))
\end{equation}
where $s_1t_1 \ldots s_nt_n = 0$ in $S$.
Clearly $J$ is an ideal of $\hat{S} \amal{E} \hat{S}$.
Consider the map that projects an element~\eqref{eq:1}
of $\hat{S} \amal{E} \hat{S}$
onto its first component $s_1\theta(t_1) \ldots
s_n\theta(t_n)$ if it is not in $J$
and to $0$ if
it is in $J$.
By the observations above, this sends $\hat{S} \amal{E} \hat{S}$
onto $S \amal{E} S$, and $S \amal{E} S \cong (\hat{S} \amal{E} \hat{S})/J$.
Since $\hat{S} \amal{E} \hat{S}$ is $E$-unitary,
it follows that $S \amal{E} S$ is strongly
$E^*$-unitary~\cite{mcalister}.
\end{proof}

\section{Examples} \label{S:calgex}

We now use our main result, Theorem~\ref{thm:Cstaramal},  and the
results of the previous section to describe the special amalgams
of various $\cstar$-algebras.

\subsection*{The Toeplitz $\cstar$-algebra}
The Toeplitz $\cstar$-algebra, which we denote $\sT$, is the
$\cstar$-subalgebra of $B(\ell^2)$ generated by the unilateral shift $S$;
see, for example, \cite[Section~V.1]{Davidson}.
It can be identified with $\cstar$-algebra of the bicyclic monoid $B$,
that is, the inverse monoid
generated by an element $a$ subject to the relation $aa^{-1}=1$, with
the semigroup homomorphism $B \to \sT$ determined by $a \mapsto S$.
The semilattice of idempotents $E = E(B)$ is a chain order-isomorphic
to the negative integers under the usual ordering.
For each element $t = a^{-i}a^j \in B$, there are only finitely
many elements of $s \in B$ such that $t \leq s$.
{}From \cite[Theorem 5.4.4]{Lawson} it is easy to see that the
non-trivial unitary full inverse submonoids of $B$ are $E(B)$ and
submonoids of the form $B(n) = \{1\} \cup \{a^{-i}a^{j} :
 i+j \equiv 0 \mod n\}$ for $n \geq 2$.
The submonoid $E(B)$ has infinitely many $\sD$-classes, while
the submonoid $B(n)$ has $n$ $\sD$-classes.

We write $\sD$ for $\cstar(E)$, the subalgebra generated by the
diagonal matrices in $\cstar(B)$.
It is isomorphic to the algebra of convergent sequences of complex numbers.
Further, $\cstar(B(n))$ can be described in several ways.  Perhaps
the simplest is the $\cstar$-subalgebra of $\sT$ generated by $S^n$ and the
$n-1$ minimal projections onto the first $n-1$ basis vectors
in $\ell^2$.

If $\sE \subset \sT$ is a $\cstar$-subalgebra of $\sT=\cstar(B)$
generated by a full submonoid $U$ of $B$, then
by Theorem \ref{thm:Cstaramal},
\[ \sT \amal{\sE} \sT = \cstar(B \amal{U} B). \]
To describe this $\cstar$-algebra, we study the
inverse semigroup structure of $B \amal{U} B$.
By Proposition~\ref{prop:E-unitary}, $B \amal{U} B$ is $E$-unitary.
For each full inverse submonoid $U$ of $B$ the semigroup $B
\amal{U} B$ is a Reilly semigroup of the form $B(G,\alpha)$ where
$G$ is the maximal subgroup of $B \amal{U} B$ containing $1$ and
$\alpha$ is some endomorphism of $G$.
The endomorphism $\alpha$ is injective since $B \amal{U} B$ is $E$-unitary.
{}From the results of~\cite{HMM}, $G$ is $F_{\infty}$, the
free group of infinite rank, if $U = E(B)$, and to $F_{n-1}$, the
free group of rank $n-1$, if $U = B(n)$.
To see this, note that the graph has two vertices (as each copy of
$B$ has one $\sD$-class) and either infinitely many edges (if $U=E$)
or $n$ edges (if $U=B(n)$); adapting the discussion before
Theorem~\ref{thm:FD} to this context gives $F_\infty$ or $F_{n-1}$,
respectively.

See \cite[Section~II.6]{Petrich} for details of structure of $B(G,\alpha)$.
Briefly, elements of $B \amal{U} B$ may be identified with triples $(i,g,j)$
where $i,j$ are positive integers and $g \in G$, with multiplication
\[ (i,g,j)(k,h,l) = \begin{cases} (i+k-j,\alpha^{k-j}(g)h,l) & \hbox{if $k \ge j$,} \\
    (i,g \alpha^{j-k}(h),l+j-k) & \hbox{if $j \ge k$.} \end{cases}
\]
An element $(i,g,j)$ of $B \amal{U} B$ can only be less than or equal
to elements of the form $(i-k,h,j-k)$ where $h \in G$ and $k \in \bN$
satisfy $i-k,j-k \geq 0$ and $\alpha^k(h)=g$.
Since $\alpha$ is injective, there is at most one such $h$.
Thus each element of $B \amal{U}
B$ has only finitely many elements above it in the natural partial
order since it is an $E$-unitary inverse semigroup.
We note that $B \amal{U} B$ is $F$-inverse, that is, each element
has a unique maximal element above it in the natural partial order.
To see this, suppose that
$(i,g,j) \leq (i-k,h,j-k), (i-l,h',j-l)$ where $l < k$.
Then $\alpha^k(h)=\alpha^l(h')$.
By the injectivity of $\alpha$, $\alpha^{k-l}(h)=h'$ and so
$(i-l,h',j-l) \leq (i-k,h,j-k)$.
It follows that $B \amal{U} B$ is $F$-inverse.

By results of Khoshkam and Skandalis~\cite{KS} (cf.~\cite{Steinberg}),
$\cstar(B \amal{U} B)$ is strongly Moria equivalent to $\cstar(E) {\times}_{\mu} H$,
a crossed product of the abelian $\cstar$-algebra $\cstar(E)$
by $H$, the maximal group homomorphic image of $B \amal{U} B$.
By Proposition~\ref{prop:universal} (or, more precisely, by Equation~\eqref{eq:amalsigma}),
if $U=E(B)$, then $G(E(B))=\{0\}$ and
$H = \bZ \amal{\{0\}} \bZ = F_2$, while if $U=B(n)$, then $G(B(n))$ is $\bZ$,
which we can identify as $n\bZ$ inside $\bZ \cong G(B)$, and so
$H = \bZ \amal{n\bZ} \bZ = \langle a,b \,|\, a^n=b^n \rangle$.
In each case, $H$ is also a semidirect product of $G$ by $\bZ$, from \cite{MR}.

To describe the action $\mu$ of $H$ on $\cstar(E)$, we start with the
Munn representation, that is,
$s \in S$ maps the set $\{ e \in E : e \le s^*s\}$ onto the set
$\{ e \in E : e \le ss^* \}$, via $e \mapsto ses^*$.

If $\hat{E}$ is the spectrum of $\cstar(E)$, then $C(\hat{E})$, the
continuous functions on $\hat{E}$, is isomorphic to $\cstar(E)$.
Moreover, $\hat{E}$ can be identified with the multiplicative linear
functionals on $E$ with the relative weak-$*$ topology.
There is a dual action of $S$ on $\hat{E}$, where $s \in S$ maps
$\{ x \in \hat{E} : x(s^*s)=1 \}$ onto $\{ x \in \hat{E} : x(ss^*)=1 \}$ via
$x \mapsto s.x$, where $s.x(e)=x(s^*es)$ for all $e \in E$.

This lifts to an action, also called $\mu$, of $S/\sigma$ on $\hat{E}$,
given by $g.x=s.x$ for any $s \in S$ with $\sigma(s)=g$ and $x(s^*s)=1$.
To see that this is well-defined, note that if $f \in E$ and $x(f)=1$, then
$s.x=(sf).x$.
By~\cite[Lemma 1.4.12]{Lawson}, for $s,t \in S$ with $\sigma(s)=\sigma(t)$,
$f=s^*st^*t$ satisfies $sf=tf$ and so $s.x=(sf).x=(tf).x=t.x$.

We summarize this discussion in the following theorem.

\begin{thm}\label{thm:Toeplitz}
If $\sD$ is the diagonal matrices in $\sT$ and $\sE=\cstar(B(n))$,
then $\sT \amal{\sD} \sT$ and $\sT \amal{\sE} \sT$ are strongly
Morita equivalent to, respectively,
\[ \sA \times_\mu F_2, \qquad \sA \times_\mu \langle a,b \,|\, a^n=b^n \rangle, \]
where $\mu$ is the action described above and $\sA$ is the algebra
of convergent sequences of complex numbers.
\end{thm}

\subsection*{Toeplitz graph $\cstar$-algebras}
Inverse semigroups associated to graphs have been defined independently
several times: \cite{AH}, \cite[Section~8.1]{Lawson99},
and \cite{Paterson02}.
We think of a (directed) graph $\Gamma$ as having a set of vertices, $\Gamma^0$, a set
of edges, $\Gamma^1$, and range and source functions $r,s:\Gamma^1\to\Gamma^0$; where the
edge $e$ goes from $s(e)$ to $r(e)$.
Define $I(\Gamma)$, the inverse semigroup associated to $\Gamma$, as
the inverse semigroup generated by $\Gamma^0 \cup \Gamma^1$ with a zero $z \notin \Gamma^0 \cup \Gamma^1$,
subject to certain relations.
Here, we use $^*$ for the inverse operation.
If we extend the source and range maps of $\Gamma^1$ to $\{ e^* : e \in \Gamma^1\}$ by
$s(e^*)=r(e)$ and $r(e^*)=s(e)$ and to $\Gamma^0$ by $s(v)=r(v)=v$, then these relations
can be conveniently summarized as
\begin{enumerate}
\item $s(e)e=e r(s)=e$ for all $e\in \Gamma^1 \cup \{ e^* : e \in \Gamma^1 \}$,
\item $ab=z$ if $a,b \in \Gamma^0 \cup \Gamma^1\cup \{ e^* : e \in \Gamma^1 \}$ with $r(a)\ne s(b)$, and
\item $a^*b=z$ if $a,b \in \Gamma^1$ and $a \ne b$.
\item $b^*b=r(b)$ if $b \in \Gamma^1$.
\end{enumerate}
Define a path in $\Gamma$ to be either a vertex, $v$, or a finite sequence of edges
$\alpha=e_1 e_2 \cdots e_n$ with $r(e_i)=s(e_{i+1})$, $1 \le i < n$.
For such a path $\alpha$, we use $\alpha^*$ for $e_n^* e_{n-1}^* \cdots e_1^*$.
Extending $s$ and $r$ to paths by $s(\alpha)=s(e_1)$ and $r(\alpha)=r(e_n)$,
there is a natural composition of paths: the product of
$\alpha$ and $\beta$ is $\alpha\beta$ if $r(\alpha)=s(\beta)$ and is $z$ otherwise.

Relations~(1) and~(2) show that any word in $\Gamma^0 \cup \Gamma^1$ must be a path
and any word in $\Gamma^0 \cup \{ e^* : e \in \Gamma^1 \}$ is $p^*$ where $p$ is a path.
Using Relation~(3), it follows that each non-zero element of $I(\Gamma)$ has the
form $pq^*$ where $p$ and $q$ are paths with $r(p)=r(q)$; further,
the product of $pq^*$ and $rs^*$ is non-zero exactly when either $q=rt$ for a path $t$
or $r=qt$ for a path $t$.  The product is either $(pt)s^*$ or $p(st)^*$, respectively.
The idempotents of $I(\Gamma)$ are the elements of the form $pp^*$ for $p$ a path.
The natural order in $I(\Gamma)$ is given by $pq^* \le rs^*$ exactly when
$p=rt$ and $q=st$ for a path $t$.

It worth observing that when $\Gamma$ is a vertex with a single edge, then $I(\Gamma)$
is the bicyclic monoid adjoin a (removable) zero, while if $\Gamma$ is a vertex
with $n$ edges, then $I(\Gamma)$ is the {\it polycyclic monoid}, that is, the monoid
generated by $n$ elements $a_1,a_2, \ldots , a_n$ subject to the relations
$a_i{a_i}^{-1} = 1, \, a_i{a_j}^{-1} = 0$ for $i \neq j$.
These monoids were introduced by Nivat and Perrot~\cite{Nivat-Perrot} in the context
of formal language theory: they were rediscovered by Renault~\cite{Renault} and are
often referred to as {\it Cuntz semigroups} in the operator algebra
literature.

Each graph inverse semigroup is $F^*$-inverse and strongly $E^*$-unitary with universal
group the free group on the edges of $\Gamma$, $F_{\Gamma^1}$~\cite{Lawson01}.
As $I(\Gamma)$ is strongly $E^*$-unitary, \cite{Milan} shows $\cstar_0(I(\Gamma))$
can be described as a partial crossed product of $\cstar_0(E(I(\Gamma)))$ by $F_{\Gamma^1}$.

This associated contracted $\cstar$-algebra is \textsl{not} the
$\cstar$-algebra of the graph, but rather the Toeplitz $\cstar$-algebra
of the graph, as defined in~\cite{FR}.
(Of course, the $\cstar$-algebra of the graph is a proper quotient of
$\cstar_0(I(\Gamma))$.)
To see this, first let $\pi : I(\Gamma) \to \cstar_0(I(\Gamma))$ be the canonical injection
of $I(\Gamma)$ in its $\cstar$-algebra and define, for $v \in \Gamma^0$, $P_v=\pi(v)$ and for
$e \in \Gamma^1$, $S_e=\pi(e)$.
Then the relations above imply that $(\{ P_v : v \in \Gamma^0\}, \{ S_e : e \in \Gamma^1\})$
form a Toeplitz-Cuntz-Krieger $\Gamma$-family and moreover, $P_v \ne 0$ for each $v$
and, if $s^{-1}(v)$ is finite, $P_v > \sum_{s(e)=v} S_e s_e^*$.
Thus, by \cite[Corollary~4.2]{FR}, $C_0^*(I(G))=C^*(\{ P_v, S_e \})$ is the Toeplitz
$\cstar$-algebra of $\Gamma$.

The $\cstar$-subalgebra of $\cstar_0(I(\Gamma))$ generated by the idempotents,
call it $\sD$, is isomorphic to $C_0(X)$, the continuous functions vanishing at infinity
on a suitable locally compact, totally disconnected, Hausdorff space $X$.
The simplest way to describe $X$ is as the space of all finite or infinite paths on $\Gamma$,
with the following topology.
A finite path $\alpha$ is closed and open if $s^{-1}(r(\alpha))$ is finite and otherwise,
has a neighborhood basis, $D_{\alpha,F}$, indexed by finite subsets $F \subset s^{-1}(r(\alpha))$.
Each $D_{\alpha,F}$ consists of paths $\alpha \beta$ where $\beta$ is finite or infinite
path with $s(\beta)=r(\alpha)$ and the first edge of $\beta$ is not in $F$.
An infinite path $\alpha$ has a neighborhood base indexed by natural numbers $n$,
$D_{\alpha,n}$ consisting of paths $\beta$ whose first $n$ edges agree with $\alpha$
and the rest are can be any edges consistent with $\beta$ being a path.

Invoking Theorem \ref{thm:Cstaramal},
\[ \cstar_0(I(\Gamma)) \amal{\sD} \cstar_0(I(\Gamma)) \cong \cstar_0(I(\Gamma) \amal{E} I(\Gamma)).  \]
By Lemma~\ref{lemma:strongamal}, $I(\Gamma) \amal{E} I(\Gamma)$ is strongly $E^*$-unitary and
by Proposition~\ref{prop:universal}, its universal group is $F_{\Gamma^1} \amal{} F_{\Gamma^1}$.
Applying Milan's Theorem~\cite[Theorem 3.3.3]{Milan} again, we have

\begin{thm}\label{thm:graph}
Let $\Gamma$ be a directed graph.
If $\sD$ is the diagonal subalgebra of $\cstar(I(\Gamma))$, then
\[ \cstar_0(I(\Gamma)) \amal{\sD} \cstar_0(I(\Gamma)) \cong \sD \times_{\mu} H, \]
where $H=F_{\Gamma^1} \amal{} F_{\Gamma^1}$ and $\mu$ is the
partial action of $H$ on $\sD$ lifted from the Munn representation.
\end{thm}

In particular, this result applies to the bicyclic monoid, so we have two
descriptions of the amalgam of the Toeplitz algebra with itself, either
as a crossed product (up to strong Morita equivalence) or as a partial
crossed product (up to $*$-isomorphism).
The theorem also applies to the Cuntz-Toeplitz algebra, when $\Gamma$ is
a vertex with $n$ loops, describing the amalgam of this algebra with
itself as a partial crossed product by $F_{2n}$, the free group of rank $2n$.

\end{document}